\documentclass[12pt]{article}
\usepackage{amssymb}
\usepackage{amsthm}
\usepackage{amsmath}
\usepackage{graphicx}
 \newtheorem{thm}{Theorem}[section]
 
 \newtheorem{lem}[thm]{Lemma}
 \newtheorem{prop}[thm]{Proposition}
 \theoremstyle{definition}
 
 \newtheorem{rem}[thm]{Remark}
 \numberwithin{equation}{section}

\begin{document}
\title{Global Solutions of Evolutionary Faddeev Model
With Small Initial Data}
\author{Zhen Lei\footnote{School of Mathematical Sciences; LMNS and Shanghai Key
  Laboratory for Contemporary Applied Mathematics, Fudan
  University, Shanghai 200433, P. R. China. {\it Email:
  leizhn@yahoo.com}}
\and Fanghua Lin\footnote{Courant Institute of Mathematics, New
  York University, USA. {\it Email: linf@cims.nyu.edu}} \and
  Yi Zhou\footnote{School of Mathematical Sciences; Shanghai Key Laboratory
  for Contemporary Applied Mathematics,  Fudan
  University, Shanghai 200433, P. R. China. {\it Email: yizhou@fudan.ac.cn}}}

\maketitle \ \ \ \ \ \ \ \ \ \ \ ---------------In Memory of
Professor Hua, Luo-Geng---------------
\maketitle
\begin{abstract}
We consider the Cauchy problem for evolutionary Faddeev model
corresponding to maps from the Minkowski space $\mathbb{R}^{1 +
n}$ to the unit sphere $\mathbb{S}^2$, which obey a system of
non-linear wave equations. The nonlinearity enjoys the null
structure and contains semi-linear terms, quasi-linear terms and
unknowns themselves. We prove that the Cauchy problem is globally
well-posed for sufficiently small initial data in Sobolev space.
\end{abstract}

\section{Introduction}

Denote an arbitrary point in $(n + 1)$-dimensional Minkowski space
$\mathbb{R} \times \mathbb{R}^n$ by $z = (t, x) = (x^\alpha)_{0
\leq \alpha \leq n}$, the space-time derivatives of a function by
$$\partial = (\partial_t, \nabla) = (\partial_\alpha)_{0 \leq \alpha \leq n}.$$
We raise and lower indices with the Minkowski metric $\eta =
(\eta_{\alpha\beta}) = \eta^{-1} = (\eta^{\alpha\beta}) = {\rm
diag}(1, - 1, - 1, - 1)$.

To describe the Faddeev model, let us consider Sobolev mappings from the
Minkowski space $(\mathbb{R} \times
\mathbb{R}^n, \eta)$, $n \geq 2$ to the unit sphere
$\mathbb{S}^2$:
\begin{equation}\label{z1}
\textbf{n}: (\mathbb{R} \times \mathbb{R}^n, \eta) \rightarrow
\mathbb{S}^2
\end{equation}
and the Lagrangian density governing the evolution of the fields
$\textbf{n}$ (see Faddeev \cite{Faddeev1, Faddeev2, Faddeev3}):
\begin{equation}\nonumber
\mathcal{L}(\textbf{n}) = \frac{1}{2}\partial_\mu\textbf{n}\cdot
\partial^\mu\textbf{n} - \frac{1}{4}(\partial_\mu\textbf{n}
\wedge \partial_\nu\textbf{n})(\partial^\mu\textbf{n} \wedge
\partial^\nu\textbf{n}).
\end{equation}
Then solutions of the Faddeev model can be characterized
variationally as critical points of the action integral
\begin{equation}\label{z2}
\mathcal{A}(\textbf{n}) = \int_{\mathbb{R} \times
\mathbb{R}^n}\mathcal{L}(\textbf{n})dxdt.
\end{equation}
The equations of motion of the Faddeev model takes the form,
\begin{equation}\label{Faddeev-1}
\textbf{n} \wedge \partial_\mu\partial^\mu\textbf{n}
  + \big[\partial_\mu\big(\textbf{n}\cdot[\partial^\mu\textbf{n} \wedge
  \partial^\nu \textbf{n}]\big)\big]\partial_\nu\textbf{n} = 0,
\end{equation}
which is the Euler-Lagrange equation of $\mathcal{L}(\textbf{n})$
 in local coordinates(see Faddeev \cite{Faddeev1, Faddeev2, Faddeev3} and Lin-Yang
\cite{Lin-Yang1} and references therein).

The Faddeev model \eqref{Faddeev-1} was introduced to model
elementary particles by using continuously extended, topologically
characterized, relativistically invariant, locally concentrated,
soliton-like fields. The model is not only important in the area of
quantum field theory but also provides many interesting and challenging
mathematical problems, see for examples \cite{RS}, \cite{S}, \cite{Wa},
\cite{Ri},\cite{M}, \cite{Es}, \cite{Cho} and \cite{MSS}.
There have been a lot of interests in recent years in studying
mathematical issues of static Faddeev model (see Lin-Yang \cite
{Lin-Yang2, Lin-Yang3, Lin-Yang4, Lin-Yang5} and review papers by
Faddeev \cite{Faddeev3} and Lin-Yang \cite{Lin-Yang1}). However,
the corresponding evolutionary equations \eqref{Faddeev-1}, which
turn out to be unusual quasi-linear wave equations, are still
untouched to our best knowledge(see also Lin-Yang \cite{Lin-Yang1}).

In the case of $n \geq 3$, there are now classical and well developed
theories on global well-posedness for quasi-linear wave equations
with null structure and small initial data, see for examples,
Christodoulou and Klainerman \cite{CK}, Lindblad-Rodnianski
\cite{Lindblad-Rhdnianski}, Sideris \cite{Sideris}.
Such theories can be easily employed to solve the evolutionary system
of the Faddeev model also when  $n \geq 3$. The aim of
this paper is to prove the global well-posedness of the Cauchy
problem of the Fadeev model \eqref{Faddeev-1} in $\mathbb{R}^{1 +
2}$ under the assumption that the initial data is small in some
generalized Sobolev space. These results provide a starting point for
further studies of evolutions of intereacting particle like approximate
solutions, see \cite{RS} and \cite{MSS}.

\begin{thm}\label{thm}
Suppose that $n_{10}, n_{20}, n_{11}, n_{21} \in
C_0^\infty(\mathbb{R}^2)$ with $s \geq 9$ and
\begin{equation}\nonumber
\|n_{10}\|_{H^{s + 2}},\ \ \|n_{20}\|_{H^{s + 2}},\ \
\|n_{11}\|_{H^{s + 1}},\ \ \|n_{21}\|_{H^{s + 1}} \leq \epsilon,\
\ n_{30} = \sqrt{1 - n_{10}^2 - n_{20}^2}.
\end{equation}
Then there exists a small positive constant $\epsilon_0$ such that
the Faddeev model \eqref{Faddeev-1} with the initial data
\begin{equation}\nonumber
n_i(0, x) = n_{i0}(x),\ \partial_t n_j(0, x) = n_{j1}(x),\ \ 1
\leq i \leq 3,\ \ 1 \leq j \leq 2
\end{equation}
is well-posed globally in time provided that $\epsilon \leq
\epsilon_0$.
\end{thm}

The nonlinearity in Faddeev model \eqref{Faddeev-1} (see also
system \eqref{Faddeev-2} in section 3 and system \eqref{Faddeev-3}
in section 4) enjoys the so-called null structure, which can be
used to explore better decay estimates of solutions, see
\cite{Christodoulou, Klainerman, Sideris, Alinhac}. In the two space
dimension, it seems the best result was due to Alinhac
\cite{Alinhac}, where the author introduced the so-called "ghost weights"
in the energies and proved a global existence result for a
class of quasi-linear wave equations (without terms which are semi-linear
and involving unknowns themselves) with small initial values and
null conditions. For quasi-linear wave equations whose nonlinearities are
cubic, and involve only the derivatives of unknowns, we refer the reader
to Li Tatsien \cite{Li} or Hoshiga \cite{Hoshiga}. One notices, however,
that the nonlinearity in the Faddeev model \eqref{Faddeev-1} contains both
semi-linear and quasi-linear terms where the semi-linear terms are
cubic  and involve the unknowns themselves (see
\eqref{Faddeev-2} in section 3). Technically, it becomes much more
complicated since the estimates for unknowns themselves can not be
obtained by the usual Klainerman's generalized energy estimates. We also
find that Alinhac's method is difficult to apply to such kind of
nonlinearities in the two space dimension.

To prove Theorem \ref{thm}, We will need the following \textit{a
priori} estimates:
\begin{thm}\label{APrioriE}
Let $s \geq 9$ and $(n_1, n_2)$ be a global classical solution to
\eqref{Faddeev-1} with initial data $n_0$ which is given in
Theorem \ref{Faddeev-1}. Then there holds
\begin{equation}\label{E}
\begin{cases}
\sum_{i = 1}^2\big\|\partial^i\textbf{n}(t, \cdot)\big\|_{\Gamma,
  s, L^2} \leq M\epsilon(1 + t)^{\delta},\\[-4mm]\\
\big\|\big(n_1(t, \cdot), n_2(t, \cdot)\big)\big\|_{\Gamma, s,
L^2}
  \leq M\epsilon(1 + t)^{\frac{1}{2} + 2\delta},\\[-4mm]\\
\big\|\big(n_1(t, \cdot), n_2(t, \cdot)\big)\big\|_{\Gamma, s - 2,
  L^\infty} \leq M\epsilon(1 + t)^{- \frac{1}{2}},
\end{cases}
\end{equation}
for some appropriately small positive constant $\delta$ and some
positive constant $M$ provided that the initial data satisfies
\begin{equation}\label{p0}
\|n_{10}\|_{H^{s + 2}},\ \ \|n_{20}\|_{H^{s + 2}},\ \
\|n_{11}\|_{H^{s + 1}},\ \ \|n_{21}\|_{H^{s + 1}} \leq \epsilon
\end{equation}
for sufficiently small positive constant $\epsilon$.
\end{thm}

To show the energy estimates of unknowns themselves (see the second
inequality in \eqref{E}), we shall use the following refined norms
as in \cite{LiYu}:
\begin{equation}\nonumber
\|f\|_{L^{p, q}} = \Big(\int_0^\infty\|f(r\xi)\|_{L^q(S^{n -
1})}^pr^{n - 1}dr\Big)^{\frac{1}{p}}.
\end{equation}
Using this norm, we are able to get essentially optimal $L^2$ estimates
for unknowns themselves:
\begin{eqnarray}\nonumber
&&\|u(t, \cdot)\|_{L^2} \leq \|u(0, \cdot)\|_{L^2} + C(1 +
  t)^{\frac{1}{2}}\Big\{\|\partial_tu(0, \cdot)\|_{L^{\frac{4}{3}}}\\\nonumber
&&\quad + \int_0^t\Big(\|\Box u(\tau, \cdot)\|_{L^{\frac{4}{3}},
\chi_1} + (1 + \tau)^{-\frac{1}{2}}\|\Box u(\tau, \cdot)\|_{L^{1,
2}, \chi_2}\Big)d\tau\Big\},
\end{eqnarray}
where $\chi_1$ is the characteristic function of $\{x: |x| \leq 1
+ \frac{\tau}{2}\}$ and $\chi_2 = 1 - \chi_1$. The proof of the
above estimate is presented in Theorem \ref{thmb22}. See
\eqref{a11} and \eqref{a12} for the definitions of the norms
appearing on the right hand side of the above inequality. The
crucial point in this \textit{a priori} estimate is that it allows
us to take the advantage of the faster time decay of $u$ in the
region of ${\rm supp\chi_1}$ and extra time decay of $u$ in the
complement of ${\rm supp\chi_1}$ which is usually due to the null
structure of nonlinearities. The above $L^2$ estimate combining
with the best $L^1-L^\infty$ estimate (see Theorem \ref{thmb12}
and Theorem \ref{thmb15}) by Klainerman \cite{Klainerman1} and
H${\rm \ddot{o}}$rmander \cite{Hormander2} allows us to be able to
get the \textit{a priori} estimate in \eqref{E}. We also point out
that our Lemma \ref{lemc13} is not covered in the Lemma 4.1 of
\cite{Alinhac}.

The analysis in this paper can be used to deal with
nonlinear wave equations with semi-linear terms, quasi-linear
terms involving unknowns themselves as well. The method
can likely be also adopted to study the sharp lifespan of nonlinear
wave equations $\Box u = F(u, \partial u, \partial^2u)$ in two and
three space dimensions with both semi-linear terms and quasi-linear
terms that may contain unknowns themselves.

The paper is organized as follows: In section 2, we review some
basic estimates for solutions of linear wave
equations and the notion of null forms. We prove then the  $L^2$ estimate
for solutions of the linear wave equations. The second and the third
\textit{a priori} estimates in the Theorem \ref{APrioriE} are established
in section 3. In the final section 4 we prove the first \textit{a
priori} estimate in Theorem \ref{APrioriE}.

\section{Preliminaries and Estimates for The Linear Wave Equations}

After some preliminary discussions, we shall prove certain energy
estimates for solutions of the linear wave equations which are
essential for establishing the inequalities in
Theorem \ref{APrioriE}. We first introduce several notations:
\begin{equation}\label{a11}
\begin{cases}
\|f\|_{L^{p, q}} = \Big(\int_0^\infty\|f(r\xi)\|_{L^q(S^{n -
1})}^pr^{n - 1}dr\Big)^{\frac{1}{p}},\\[-4mm]\\
\|f\|_{L^{\infty, q}} = \sup_{r \geq 0}\|f(r\xi)\|_{L^q(S^{n -
1})}.
\end{cases}
\end{equation}
It is easy to see that
\begin{equation}\nonumber
\|f\|_{L^{p, p}} = \|f\|_{L^p}.
\end{equation}

For any integer $s \geq 0$, real numbers $1 \leq p, q \leq \infty$
and any characteristic function $\psi(t, x)$, we will denote
\begin{equation}\label{a12}
\begin{cases}
\|u(t, \cdot)\|_{\Gamma, s, L^{p, q}, \psi} =
  \sum_{|k| \leq s}\|\psi(t, \cdot)\Gamma^ku(t, \cdot)\|_{L^{p, q}},\\[-4mm]\\
\|u(t, \cdot)\|_{\Gamma, s, L^{p, q}} =
  \sum_{|k| \leq s}\|\Gamma^ku(t, \cdot)\|_{L^{p, q}}.
\end{cases}
\end{equation}

Here as in Klainerman \cite{Klainerman1},  we use the
following vector fields(operators):
\begin{equation}\label{c1-1}
\Gamma = (\partial, L, \Omega),
\end{equation}
where
\begin{equation}\nonumber
\begin{cases}
\partial = (\partial_\alpha)_{0 \leq \alpha \leq n}
  = (\partial_t, \nabla),\\[-4mm]\\
\Omega = (\Omega_{ij})_{1 \leq i, j \leq n,\ i \neq j},
  \quad \Omega_{ij} = x_i\partial_{j} - x_j\partial_{i},\\[-4mm]\\
L_0 = t\partial_t + r\partial_r = x_\alpha\partial_\alpha,\quad
L_i = t\partial_{i} + x_i\partial_t.
\end{cases}
\end{equation}
Denote the wave operator by $\Box = \partial_t^2 - \Delta$ and the
Possion product by $[\cdot, \cdot]$.

First of all, it is easy to check that the following Proposition
holds:
\begin{prop}\label{propa11}
For any multi-index $\alpha$, we have
\begin{equation}\label{a13}
[\Box, \Gamma^\alpha] = \sum_{|\beta| \leq |\alpha| -
1}C_{\alpha\beta}\Gamma^\beta\Box,\quad [\partial, \Gamma^\alpha]
= \sum_{|\beta| \leq |\alpha| -
1}C_{\alpha\beta}^\prime\Gamma^\beta\partial
\end{equation}
for some constants $C_{\alpha\beta}$ and $C_{\alpha\beta}^\prime$.
\end{prop}

Concerning Klainerman's vector fields in \eqref{c1-1}, one also
has following Proposition:
\begin{prop}\label{propa12}
There exists a positive constant $C$ such that
\begin{equation}\label{a14}
|\partial^su(t, x)| \leq C\big(1 + \big|t-
|x|\big|\big)^{-s}\sum_{|\alpha| \leq s}|\Gamma^\alpha u(t, x)|
\end{equation}
holds for all smooth function $u(t, x)$.
\end{prop}
\begin{proof}
In fact, \eqref{a14} is obvious if $\big|t- |x|\big| \leq 1$.
Otherwise, \eqref{a14} follows from the following expressions:
\begin{equation}\nonumber
\partial_t = \frac{tL_0 - x_iL_i}{t^2 - |x|^2},\quad
\partial_{x_j} = \frac{tL_j - x_jL_0 - x_k\Omega_{kj}}{t^2 -
|x|^2}.
\end{equation}
\end{proof}

Next let us recall the $L^\infty-L^1$ estimate for linear wave
equations, whose proof can be found in Klainerman
\cite{Klainerman3}.
\begin{thm}\label{thmb12}
Assume that $u$ solves the Cauchy problem of the homogeneous
linear wave equation in $\mathbb{R} \times \mathbb{R}^n$:
\begin{equation}\label{b111-1}
\Box u = 0,\quad u(0, x) = u_0(x), u_t(0, x) = u_1(x).
\end{equation}
Then we have
\begin{equation}\label{b16}
\|u(t, \cdot)\|_{L^\infty} \leq \frac{C\big(\|u_0\|_{W^{n, 1}} +
\|u_1\|_{W^{n - 1, 1}}\big)}{(1 + t)^{\frac{n - 1}{2}}}
\end{equation}
for all $t \geq 0$.
\end{thm}

The following estimate can be found in H${\rm \ddot{o}}$rmander
\cite{Hormander2}.
\begin{thm}\label{thmb15}
Let $u$ solve the Cauchy problem of the inhomogeneous linear wave
equation in $\mathbb{R} \times \mathbb{R}^2$:
\begin{equation}\label{b111}
\Box u = f,\quad u(0, x) = u_t(0, x) = 0.
\end{equation}
Then we have
\begin{eqnarray}\label{b112}
|u(t, x)| \leq \frac{C\int_0^t\|f(\tau, \cdot)\|_{\Gamma, 1,
L^1}(1 + \tau)^{-\big(\frac{1}{2} - l\big)}d\tau}{(1 + t +
|x|)^{\frac{1}{2}}\big(1 + \big|t - |x|\big|\big)^l}.
\end{eqnarray}
Here $0 \leq l \leq \frac{1}{2}$.
\end{thm}

We will need some Sobolev type inequalities.  The first one is the
well-known Sobolev Imbedding theorem on the unit sphere
$\mathbb{S}^{n - 1}$ centered at the origin:
\begin{thm}\label{thma21}
Let $x = r\xi$, $r = |x|$. Then there holds
\begin{equation}\label{a21}
\begin{cases}
sp > n - 1: |v(x)| = |v(r\xi)| \leq C\sum_{|k| \leq
  s}\|\Omega^kv(r\xi)\|_{L^p_\xi},\\[-4mm]\\
sp < n - 1: \|v(r\xi)\|_{L^q_\xi} \leq C\sum_{|k| \leq
  s}\|\Omega^kv(r\xi)\|_{L^p_\xi},\\ \quad\quad\quad\quad\quad\quad
  \frac{1}{q} = \frac{1}{p} - \frac{s}{n - 1},\\[-4mm]\\
sp = n - 1: \|v(r\xi)\|_{L^q_\xi} \leq C\sum_{|k| \leq
  s}\|\Omega^kv(r\xi)\|_{L^p_\xi},\quad  p \leq q < \infty
\end{cases}
\end{equation}
for all smooth function $v(x)$.
\end{thm}

The second one is the Sobolev Imbedding theorem in a ball
$\mathbf{B}_\lambda$ with radius $\lambda$ centered at the origin:
\begin{thm}\label{thma22}
Let $\lambda > 0$. Then there exists a positive constant $C$
independent of $\lambda$ such that
\begin{equation}\label{a22}
\begin{cases}
sp > n: \|u\|_{L^\infty(B_\lambda)} \leq
  C\lambda^{-\frac{n}{p}}\sum_{|\alpha| \leq s}\lambda^{|\alpha|}
  \|\nabla^\alpha u\|_{L^p(B_\lambda)},\\[-4mm]\\
sp < n: \|u\|_{L^q(B_\lambda)} \leq C
  \lambda^{-n\big(\frac{1}{p} - \frac{1}{q}\big)}
  \sum_{|k| \leq s}\lambda^{|\alpha|}\|\nabla^\alpha u\|_{L^p(B_\lambda)},\\
  \quad\quad\quad\quad\quad\quad q = \frac{np}{(n - sp)},
  \frac{1}{q} = \frac{1}{p} - \frac{s}{n},\\[-4mm]\\
sp = n: \|u\|_{L^q(B_\lambda)} \leq C
  \lambda^{-n\big(\frac{1}{p} - \frac{1}{q}\big)} \sum_{|k| \leq
  s}\lambda^{|\alpha|}\|\nabla^\alpha u\|_{L^p(B_\lambda)},\\
  \quad\quad\quad\quad\quad\quad  p \leq q < \infty.
\end{cases}
\end{equation}
\end{thm}
\begin{proof}
For $\lambda = 1$, these are standard Sobolev imbedding
inequalities. When $\lambda \neq 1$, it follows from a simple
scaling technique.
\end{proof}

Let us improve \eqref{a22} to get decaying type inequalities
for smooth function $u(t, x)$ using the norms defined in
\eqref{a11} and \eqref{a12}.

\begin{lem}\label{lema32}
Let $\chi_1$ be the characteristic function of $\big\{x\big| |x|
\leq 1 + \frac{t}{2}\big\}$. Then
\begin{equation}\label{a32}
\begin{cases}
sp > n: \|u(t, \cdot)\|_{L^\infty, \chi_1} \leq C(1 + t)^{-
  \frac{n}{p}}\|u(t, \cdot)\|_{\Gamma, s, L^p, \chi_1}\\[-4mm]\\
sp < n: \|u(t, \cdot)\|_{L^q, \chi_1} \leq C(1 + t)^{- n
  \big(\frac{1}{p} - \frac{1}{q}\big)}\|u(t, \cdot)\|_{\Gamma, s, L^p, \chi_1}\\
  \quad\quad\quad\quad\quad \frac{1}{q} = \frac{1}{p} - \frac{s}{n},\\[-4mm]\\
sp = n: \|u(t, \cdot)\|_{L^q, \chi_1} \leq C(1 + t)^{- n
  \big(\frac{1}{p} - \frac{1}{q}\big)}\|u(t, \cdot)\|_{\Gamma, s, L^p, \chi_1}\\
  \quad\quad\quad\quad\quad p \leq q < \infty.
\end{cases}
\end{equation}
\end{lem}
\begin{proof}
Letting $\lambda = \frac{t}{2} + 1$ in Theorem \ref{thma22}, and
then using Proposition \ref{propa12}, one can easily check the
above decaying type inequalities.
\end{proof}

The following Lemma involves the estimate of Sobolev norms for
composite functions, which can be easily proved by chain rules and
H${\rm \ddot{o}}$lder inequality.
\begin{lem}\label{lemc12}
Let $\alpha$ be a non-negative integer and $F$ be a smooth
function with $F(w) = O(|w|^{1 + \alpha})$ for $|w| \leq 1$. For
any integer $s \geq 0$ and any characteristic function $\chi$,
there exists a positive constant $C$ such that
\begin{equation}\label{c13}
\begin{cases}
\alpha = 0: \|F(w(t, \cdot))\|_{\Gamma, s, L^{p, q}, \chi} \leq
  C\|w(t, \cdot)\|_{\Gamma, s, L^{p, q}, \chi},\\[-4mm]\\
\alpha \geq 1: \|F(w(t, \cdot))\|_{\Gamma, s, L^{p, q}, \chi} \leq
  C\prod_{i = 1}^\alpha\|w(t, \cdot)\|_{\Gamma, s, L^{p_i, q_i}, \chi}
  \|w(t, \cdot)\|_{\Gamma, s, L^{p_0, q_0}, \chi},\\
\quad\quad\quad \frac{1}{p} = \sum_{i =
0}^\alpha\frac{1}{p_i},\quad \frac{1}{q} = \sum_{i =
0}^\alpha\frac{1}{q_i}
\end{cases}
\end{equation}
holds for all $w$ with $\|w(t, \cdot)\|_{\Gamma, [\frac{s}{2}],
L^\infty} \leq 1$.
\end{lem}

Finally, let us recall the definition of null structure satisfied
by nonlinearity in nonlinear wave equations. For $0 \leq \alpha,
\beta \leq n$, let
\begin{equation}\nonumber
Q_{\alpha\beta}(f, g) = \partial_\alpha f\partial_\beta g -
\partial_\alpha f\partial_\beta g,\quad Q(f, g) =
\partial_tf\partial_tg - (\nabla f)(\nabla g).
\end{equation}
$Q_{\alpha\beta}(f, g)$ and $Q(f, g)$ are called nonlinearities
with null structure. Concerning the nonlinearities with null
structure, we have
\begin{lem}\label{lemd11}
Let $Q_{\alpha\beta}(f, g)$ and $Q(f, g)$ are nonlinearities with
null structure. Then one has
\begin{eqnarray}\label{d18}
|Q_{\alpha\beta}(f, g)(t, x)| + |Q(f, g)(t, x)| \leq
\frac{C\big(|\Gamma f||Dg| + |Df||\Gamma g|\big)}{1 + t}.
\end{eqnarray}
\end{lem}
\begin{proof}
In fact, one can check the identities
\begin{equation}\nonumber
\begin{cases}
Q_{ij}(f, g) = \frac{-\partial_tf\Omega_{ij}g + L_if\partial_jg -
L_jf\partial_ig}{t},\\[-4mm]\\
Q_{0j}(f, g) = \frac{\partial_tfL_jg - L_jf\partial_tg}{t},\\[-4mm]\\
Q(f, g) = \frac{\partial_tfL_0g - \sum_{i =
1}^2L_if\partial_ig}{t},
\end{cases}
\end{equation}
which give \eqref{d18} if $t$ is large. In the case that $t$ is
small, \eqref{d18} is obvious.
\end{proof}

It is easy to check the following commutating property (see
Klainerman \cite{Klainerman}):
\begin{lem}\label{lemd12}
Let $\Gamma$ be any vector field defined in \eqref{c1-1},
$Q_{\alpha\beta}(f, g)$ and $Q(f, g)$ be the nonlinearities with
null structure as in Lemma \ref{lemd11}. Then there holds
\begin{equation}\nonumber
\begin{cases}
[\Gamma, Q_{\alpha\beta}] = \lambda^{\alpha\beta\gamma\delta}Q_{\gamma\delta},\\[-4mm]\\
[\Gamma, Q] = \lambda Q_{\gamma\delta},
\end{cases}
\end{equation}
where $\lambda's$ are constants and $[\Gamma, Q_{\alpha\beta}](f,
g) = \Gamma Q(f, g) - Q(\Gamma f, g) - Q(f, \Gamma g)$.
\end{lem}

\section{Proof of Theorem \ref{thm}}

Let us first rewrite the Faddeev model \eqref{Faddeev-1}
under geodesic normal coordinates $(n_1, n_2)$:
\begin{eqnarray}\label{Faddeev-2}
&&\partial_\mu\partial^\mu\begin{pmatrix}n_1\\ n_2
  \end{pmatrix} + \frac{\partial_\mu n_1\partial^\mu n_1
  + \partial_\mu n_2\partial^\mu n_2}{1 - n_1^2 - n_2^2}
  \begin{pmatrix}n_1\\ n_2\end{pmatrix}\\\nonumber
&&\quad -\ \frac{n_2^2\partial_\mu n_1\partial^\mu n_1 +
  n_1^2\partial_\mu n_2\partial^\mu n_2 - 2n_1n_2\partial_\mu
  n_1\partial^\mu n_2}{1 - n_1^2 - n_2^2}
  \begin{pmatrix}n_1\\ n_2\end{pmatrix}\\\nonumber
&&\quad +\ \frac{\partial_\mu\Big(
  \frac{\partial^\mu n_1\partial^\nu n_2 - \partial^\nu
  n_1\partial^\mu n_2}{\sqrt{1 - n_1^2 - n_2^2}}\Big)}
  {\sqrt{1 - n_1^2 - n_2^2}}
  \begin{pmatrix}(1 - n_1^2)\partial_\nu n_2
  + n_1n_2\partial_\nu n_1\\ -  (1 - n_2^2)\partial_\nu n_1
  - n_1n_2\partial_\nu n_2\end{pmatrix} = 0,
\end{eqnarray}
which turns out to be quasi-linear wave equations. The local
existence of classical solutions for quasilinear wave equations is
well-known provided that the initial data belongs to Sobolev space
$H^{s + 2} \times H^{s + 1}$ with $s \geq 1$ (see
\cite{Hughes-Kato-Marsden}). Consequently, our main Theorem
\ref{thm} is just a corollary of the \textit{a priori} estimates
\eqref{E} in Theorem \ref{APrioriE}.

This section and section 4 are devoted to establishing the
\textit{a priori} estimates in Theorem \ref{APrioriE}. Our
strategy is to use the continuity arguement in the time variable
$t$. By \cite{Hughes-Kato-Marsden} and the assumptions on the
initial data in \eqref{p0}, it is obvious that \eqref{E} is true
for sufficiently small time $t$ and some big constant $M$
depending only on the norms of the initial data in Theorem
\ref{thm}. Let us assume that $T > 0$ is the biggest time such
that \eqref{E} is true on $0 \leq t \leq T$. If $T = \infty$, then
we are done. If $T < \infty$, we are going to prove that
\begin{equation}\label{Ea}
\begin{cases}
\sum_{i = 1}^2\big\|\big(\partial^in_1(t, \cdot), \partial^in_2(t,
  \cdot)\big)\big\|_{\Gamma, s, L^2} < M\epsilon(1 + t)^{\delta},\\[-4mm]\\
\big\|\big(n_1(t, \cdot), n_2(t, \cdot)\big)\big\|_{\Gamma, s,
  L^2} < M\epsilon(1 + t)^{\frac{1}{2} + 2\delta},\\[-4mm]\\
\big\|\big(n_1(t, \cdot), n_2(t, \cdot)\big)\big\|_{\Gamma, s - 2,
  L^\infty} < M\epsilon(1 + t)^{- \frac{1}{2}},
\end{cases}
\end{equation}
${\rm for}\ 0 \leq t \leq T$. By \cite{Hughes-Kato-Marsden} again, we
conclude that \eqref{E} is valid at least for $0 \leq t \leq T + \delta_0$
with a sufficiently small $\delta_0 > 0$, and hence we obtain a
contradiction to the maximality of $T$.Thus \eqref{E} is valid for all
time $t \geq 0$.

Consequently, our goal is to prove that \eqref{Ea} is true for $0
\leq t \leq T$ under the assumption that \eqref{E} is true for $0
\leq t \leq T < \infty$. Before doing that, let us prove the
following Lemma concerning the following $L^2$ estimate of the
unknown itself for linear wave equation:
\begin{thm}\label{thmb22}
Assume that $u$ solve the linear wave equation in $\mathbb{R}
\times \mathbb{R}^2$:
\begin{equation}\label{b111-2}
\Box u = f,\quad u(0, x) = u_0(x), u_t(0, x) = u_1(x).
\end{equation}
Then we have
\begin{eqnarray}\label{b24}
&&\|u(t, \cdot)\|_{L^2} \leq \|u_0(\cdot)\|_{L^2} + C(1 +
  t)^{\frac{1}{2}}\Big\{\|u_1(\cdot)\|_{L^{\frac{4}{3}}}\\\nonumber
&&\quad + \int_0^t\Big(\|f(\tau, \cdot)\|_{L^{\frac{4}{3}},
\chi_1} + (1 + \tau)^{-\frac{1}{2}}\|f(\tau, \cdot)\|_{L^{1, 2},
\chi_2}\Big)d\tau\Big\},
\end{eqnarray}
where $\chi_1$ is the characteristic function of $\{x: |x| \leq 1
+ \frac{t}{2}\}$ and $\chi_2 = 1 - \chi_1$.
\end{thm}
\begin{proof}
To prove \eqref{b24}, we compute that
\begin{eqnarray}\nonumber
&&\|u(t, \cdot)\|_{L^2} \leq \|u_0(\cdot)\|_{L^2} +
  \Big\|\frac{\sin(|\xi|t)}{|\xi|}\widehat{u_1}(\xi)\Big\|_{L^2}\\\nonumber
&&\quad +\ \int_0^t\Big\|\frac{\sin\big(|\xi|(t - \tau)
  \big)}{|\xi|}\widehat{\chi_1f}(\tau, \xi)\Big\|d\tau
  + \int_0^t\Big\|\frac{\sin\big(|\xi|(t - \tau)
  \big)}{|\xi|}\widehat{\chi_2f}(\tau, \xi)\Big\|d\tau.
\end{eqnarray}
Next we do the following straightforward computation
\begin{eqnarray}\nonumber
&&\Big\|\frac{\sin(|\xi|t)}{|\xi|}\widehat{u_1}(\xi)\Big\|_{L^2}
  = \Big\|\frac{\sin(|\eta|)}{|\eta|}\widehat{u_1}
  \big(\frac{\eta}{t}\big)\Big\|_{L^2}\\\nonumber
&&\leq C\Big\|(1 + |\eta|)^{-1}
  \widehat{u_1}\big(\frac{\eta}{t}\big)\Big\|_{L^2}
  =  C t^2\Big\|(1 + |\eta|)^{-1}
  \widehat{u_1(tx)}(\eta)\Big\|_{L^2}\\\nonumber
&&\leq Ct^2\|u_1(tx)\|_{H^{-1}} \leq Ct^2
  \|u_1(tx)\|_{L^{\frac{4}{3}}}
  = Ct^{\frac{1}{2}}\|u_1\|_{L^{\frac{4}{3}}}.
\end{eqnarray}
A similar estimate also holds for
$\int_0^t\Big\|\frac{\sin\big(|\xi|(t - \tau)
\big)}{|\xi|}\widehat{\chi_1f}(\tau, \xi)\Big\|d\tau$.

Finally we compute
\begin{eqnarray}\nonumber
&&\Big\|\frac{\sin\big(|\xi|(t - \tau)
  \big)}{|\xi|}\widehat{\chi_2f}(\tau, \xi)\Big\|\\\nonumber
&&\leq C(t - \tau)^2\|\chi_2f\big(\tau, (t - \tau)x
  \big)\|_{H^{-1}}\\\nonumber
&&= C(t - \tau)^2\sup_{v\in H^1}\frac{\int\chi_2f\big(\tau, (t -
  \tau)x\big)v(x)dx}{\|v\|_{H^1}}\\\nonumber
&&\leq C(t - \tau)^2\sup_{v\in H^1}\frac{\big\|\chi_2f\big(\tau,
  (t - \tau)x\big)\big\|_{L^{1, 2}}\|v\|_{L^{\infty, 2}, |(t - \tau)y|
  \geq \frac{1 + \tau}{2}}}{\|v\|_{H^1}}\\\nonumber
&&\leq C(t - \tau)^2\big\|\chi_2f\big(\tau,
  (t - \tau)x\big)\big\|_{L^{1, 2}}\sup_{v\in H^1}\frac{\Big(\sup_{r
  \geq \frac{1 + \tau}{2(t - \tau)}}\int |v(r\xi)|^2d\xi\Big)^{\frac{1}{2}}}
  {\|v\|_{H^1}}\\\nonumber
&&\leq C\|f(\tau, x)\|_{L^{1, 2}, \chi_2}\sup_{v\in H^1}
  \frac{\Big(\sup_{r \geq \frac{1 + \tau}{2(t - \tau)}}-\int_r^\infty\partial_r\int
  |v(r\xi)|^2d\xi dr\Big)^{\frac{1}{2}}}{\|v\|_{H^1}}\\\nonumber
&&\leq C\|f(\tau, x)\|_{L^{1, 2}, \chi_2}\sqrt{\frac{t - \tau}{1 +
  \tau}} \leq C\sqrt{\frac{1 + t}{1 +
  \tau}}\|f(\tau, x)\|_{L^{1, 2}, \chi_2}.
\end{eqnarray}
The proof of Theorem \ref{thmb22} is thus completed.
\end{proof}

\begin{rem}\label{rem1}
It is easy to see that the following estimate
\begin{eqnarray}\label{b24-1}
\|u(t, \cdot)\|_{L^2} \leq \|u_0(\cdot)\|_{L^2} + C(1 +
t)^{\frac{1}{2}}\Big\{\|u_1(\cdot)\|_{L^{\frac{4}{3}}} +
\int_0^t\|f(\tau, \cdot)\|_{L^{\frac{4}{3}}}d\tau\Big\}
\end{eqnarray}
is also true by using the same proof as that for Theorem \ref{thmb22}.
In fact, one deduces \eqref{b24-1} in the case that $f(t, x)$
decays sufficiently fast outside the light cone $\{(t, x):
|x| \leq 1 + \frac{t}{2}\}$.
\end{rem}

Now let us move on to show \eqref{Ea} for $0 \leq t \leq T$
under the assumption that \eqref{E} is true for $0 \leq t \leq T <
\infty$. We shall first prove the second and third \textit{a
priori} estimates in \eqref{Ea}, while we will prove the first inequality
of \eqref{Ea} in section 4.

First of all, noting that $[\frac{s}{2}] + 5 \leq s$ for $s \geq
9$, one can easily deduce from Sobolev inequality and the third
inequality in \eqref{E} that
\begin{eqnarray}\label{p1-1}
&&\sum_{i = 0}^2\big\|\big(\partial^in_1(t, \cdot),
  \partial^in_2(t, \cdot)\big)\big|_{\Gamma,
  [\frac{s}{2}] + 1, L^\infty}\\\nonumber
&&\leq C\big\|\big(n_1(t, \cdot), n_2(t, \cdot)\big)\big\|_{
  \Gamma, [\frac{s}{2}] + 3, L^\infty}\leq C\big\|\big(n_1(t, \cdot), n_2(t, \cdot)
  \big)\big\|_{\Gamma,  s - 2, L^\infty}\\\nonumber
&&\leq CM\epsilon(1 + t)^{- \frac{1}{2}}.
\end{eqnarray}
Inequality \eqref{p1-1} will be used repeatly in the rest of this
section and in section 4.

\bigskip

\noindent \textbf{Estimates for $\big\|\big(n_1(t, \cdot), n_2(t,
\cdot)\big)\big\|_{\Gamma, s - 2, L^\infty}$ in \eqref{Ea}}

\bigskip

Let
\begin{eqnarray}\label{p1-3}
&&f = \begin{pmatrix}f_1\\ f_2
  \end{pmatrix} = - \frac{\partial_\mu n_1\partial^\mu n_1
  + \partial_\mu n_2\partial^\mu n_2}{1 - n_1^2 - n_2^2}
  \begin{pmatrix}n_1\\ n_2\end{pmatrix}\\\nonumber
&&\quad +\ \frac{n_2^2\partial_\mu n_1\partial^\mu n_1 +
  n_1^2\partial_\mu n_2\partial^\mu n_2 - 2n_1n_2\partial_\mu
  n_1\partial^\mu n_2}{1 - n_1^2 - n_2^2}
  \begin{pmatrix}n_1\\ n_2\end{pmatrix}\\\nonumber
&&\quad -\ \frac{\partial_\mu\Big(
  \frac{\partial^\mu n_1\partial^\nu n_2 - \partial^\nu
  n_1\partial^\mu n_2}{\sqrt{1 - n_1^2 - n_2^2}}\Big)}
  {\sqrt{1 - n_1^2 - n_2^2}}
  \begin{pmatrix}(1 - n_1^2)\partial_\nu n_2
  + n_1n_2\partial_\nu n_1\\ -  (1 - n_2^2)\partial_\nu n_1
  - n_1n_2\partial_\nu n_2\end{pmatrix}.
\end{eqnarray}
By Proposition \ref{propa11}, one has
\begin{equation}\label{p1-4}
\begin{cases}
\Box\Gamma^\alpha n_1 = \sum_{\beta \leq \alpha}
  C_{\beta}\Gamma^\beta f_1,\\[-4mm]\\
\Box\Gamma^\alpha n_2 = \sum_{\beta \leq \alpha}
  C_{\beta}\Gamma^\beta f_2.
\end{cases}
\end{equation}
Consequently, by Theorem \ref{thmb12} and Theorem \ref{thmb15}, we
have
\begin{eqnarray}\label{p2-1}
&&\|n_1(t, \cdot)\|_{\Gamma, s - 2, L^\infty} \leq
  C\sum_{|\alpha| \leq s - 2}\|\Gamma^\alpha n_1(t,
  \cdot)\|_{L^\infty}\\\nonumber
&&\leq C(1 + t)^{-\frac{1}{2}}\sum_{|\alpha| \leq s - 2}
  \Big\{\|\Gamma^\alpha n_1(0, \cdot)\|_{W^{2, 1}} +
  \|\partial_t\Gamma^\alpha n_1(0, \cdot)\|_{W^{1,
  1}}\\\nonumber
&&\quad +\ \int_0^t\big[\|\Gamma^\alpha f_1(\tau, \cdot)
  \|_{W^{1, 1}}(1 +
  \tau)^{-\frac{1}{2}}\big]d\tau\Big\}\\\nonumber
&&\leq C(1 + t)^{-\frac{1}{2}}\Big\{\|(n_{10}, n_{20})\|_{H^{s +
  2}} + \|(n_{11}, n_{21})\|_{H^{s + 1}}\\\nonumber
&&\quad +\ \int_0^t\big[\|f_1(\tau, \cdot)\|_{\Gamma, s - 1,
  L^1}(1 + \tau)^{-\frac{1}{2}}\big]d\tau\Big\}.
\end{eqnarray}
Here we point out that one can use equations \eqref{p1-3} to
express $\Gamma^\alpha n_1$, $\partial_t(\Gamma^\alpha
n_1)$, $\Gamma^\alpha n_2$ and $\partial_t(\Gamma^\alpha n_2)$ at
time $t = 0$ in terms of the spacial derivatives of $n_{10}$,
$n_{20}$, $n_{11}$ and $n_{21}$. As a consequence, one has that
\begin{eqnarray}\nonumber
&&\sum_{|\alpha| \leq s - 2}\Big\{\|\Gamma^\alpha n_1(0,
  \cdot)\|_{W^{2, 1}} + \|\partial_t\Gamma^\alpha n_1(0,
  \cdot)\|_{W^{1, 1}}\Big\}\\\nonumber
&&\leq C\big(\|(n_{10}, n_{20})\|_{H^{s + 2}} + \|(n_{11},
  n_{21})\|_{H^{s + 1}}\big).
\end{eqnarray}
Repeating the above argument, one also has
\begin{eqnarray}\label{p2-2}
&&\|n_2(t, \cdot)\|_{\Gamma, s - 2, L^\infty} \leq
  C(1 + t)^{-\frac{1}{2}}\Big\{\|(n_{10}, n_{20})\|_{H^{s +
  2}} + \|(n_{11}, n_{21})\|_{H^{s + 1}}\\\nonumber
&&\quad +\ \int_0^t\big[\|f_2(\tau, \cdot)\|_{\Gamma, s - 1,
  L^1}(1 + \tau)^{-\frac{1}{2}}\big]d\tau\Big\}.
\end{eqnarray}
To proceed further, we need estimate $\|f_1(\tau, \cdot)\|_{\Gamma, s - 1,
L^1}$
and $\|f_2(\tau, \cdot)\|_{\Gamma, s - 1, L^1}$.

First of all, H${\rm \ddot{o}}$lder inequality gives
\begin{eqnarray}\label{p3-1}
&&\Big\|\frac{n_1\big(\partial_\mu n_1\partial^\mu n_1
  + \partial_\mu n_2\partial^\mu n_2\big)}{1 - n_1^2 - n_2^2}
  \Big\|_{\Gamma, s - 1, L^1}\\\nonumber
&&\leq C\Big\|\frac{n_1}{1 - n_1^2 - n_2^2}\Big\|_{\Gamma, s
  - 1, L^2}\big\|\partial_\mu n_1\partial^\mu n_1 + \partial_\mu
  n_2\partial^\mu n_2\big\|_{\Gamma, [\frac{s - 1}{2}], L^2}\\\nonumber
&&\quad  +\ C\Big\|\frac{n_1}{1 - n_1^2 - n_2^2}\Big\|_{\Gamma,
  [\frac{s - 1}{2}], L^\infty}\big\|\partial_\mu n_1\partial^\mu n_1
  + \partial_\mu n_2\partial^\mu n_2\big\|_{\Gamma, s - 1,
  L^1}.
\end{eqnarray}
By \eqref{E} and Lemma \ref{lemc12}, we have
\begin{eqnarray}\nonumber
\Big\|\frac{n_1}{1 - n_1^2 - n_2^2}\Big\|_{\Gamma, s
  - 1, L^2} \leq C\big\|(n_1, n_2)\big\|_{\Gamma, s -
1, L^2} \leq CM\epsilon(1 + \tau)^{\frac{1}{2} + 2\delta}.
\end{eqnarray}
By Lemma \ref{lemd11} and Lemma \ref{lemd12}, we estimate
\begin{eqnarray}\nonumber
&&\big\|\partial_\mu n_1\partial^\mu n_1 + \partial_\mu
  n_2\partial^\mu n_2\big\|_{\Gamma, [\frac{s - 1}{2}],
  L^2}\\\nonumber
&&\leq C\sum_{|\alpha + \beta| \leq [\frac{s -
  1}{2}]}\big\|\partial_\mu\Gamma^\alpha n_1\partial^\mu
  \Gamma^\beta n_1 + \partial_\mu\Gamma^\alpha n_2\partial^\mu
  \Gamma^\beta n_2\big\|_{L^2}\\\nonumber
&&\leq C(1 + \tau)^{-1}\sum_{|\alpha + \beta| \leq [\frac{s -
  1}{2}]}\big\|\partial\Gamma^\alpha n_1\Gamma
  \Gamma^\beta n_1 + \partial\Gamma^\alpha n_2\Gamma
  \Gamma^\beta n_2\big\|_{L^2}\\\nonumber
&&\leq C(1 + \tau)^{-1}\|(\partial n_1, \partial n_2)\|_{\Gamma, s
  - 3, L^2}\|(n_1, n_2)\|_{\Gamma, s - 2, L^\infty}\\\nonumber
&&\leq C(M\epsilon)^2(1 + \tau)^{- \frac{3}{2} + \delta}.
\end{eqnarray}
Consequently, the above two estimates yield
\begin{eqnarray}\label{p3-2}
\Big\|\frac{n_1}{1 - n_1^2 - n_2^2}\Big\|_{\Gamma, s
  - 1, L^2}\big\|\partial_\mu n_1\partial^\mu n_1 + \partial_\mu
  n_2\partial^\mu n_2\big\|_{\Gamma, [\frac{s - 1}{2}], L^2}
\leq C(M\epsilon)^3(1 + \tau)^{- 1 + 3\delta}.
\end{eqnarray}
 Similarly, using \eqref{E} and Lemma \ref{lemc12} onc more time, we can
deduce that
\begin{eqnarray}\nonumber
\Big\|\frac{n_1}{1 - n_1^2 - n_2^2}\Big\|_{\Gamma,
  [\frac{s - 1}{2}], L^\infty} \leq
CM\epsilon(1 + \tau)^{- \frac{1}{2}}.
\end{eqnarray}
By \eqref{E}, Lemma \ref{lemd11} and Lemma \ref{lemd12}, one thus has
\begin{eqnarray}\nonumber
&&\big\|\partial_\mu n_1\partial^\mu n_1
  + \partial_\mu n_2\partial^\mu n_2\big\|_{\Gamma, s - 1,
  L^1}\\\nonumber
&&\leq C(1 + \tau)^{-1}\|(\partial n_1, \partial n_2)\|_{\Gamma, s
  - 1, L^2}\|(n_1, n_2)\|_{\Gamma, s, L^2}\\\nonumber
&&\leq C(M\epsilon)^2(1 + \tau)^{- \frac{1}{2} + 3\delta}.
\end{eqnarray}
Consequently, one obtains that
\begin{eqnarray}\label{p3-3}
&&\Big\|\frac{n_1}{1 - n_1^2 - n_2^2}\Big\|_{\Gamma,
  [\frac{s - 1}{2}], L^\infty}\big\|\partial_\mu n_1\partial^\mu n_1
  + \partial_\mu n_2\partial^\mu n_2\big\|_{\Gamma, s - 1,
  L^1}\\\nonumber
&&\leq C(M\epsilon)^3(1 + \tau)^{- 1 + 3\delta}.
\end{eqnarray}
Combining \eqref{p3-2}, \eqref{p3-3} with \eqref{p3-1}, we thus conclude
\begin{equation}\label{p3-4}
\Big\|\frac{n_1\big(\partial_\mu n_1\partial^\mu n_1
  + \partial_\mu n_2\partial^\mu n_2\big)}{1 - n_1^2 - n_2^2}
  \Big\|_{\Gamma, s - 1, L^1} \leq C(M\epsilon)^3(1 + \tau)^{- 1 + 3\delta}.
\end{equation}

A similar argument gives also that
\begin{eqnarray}\label{p3-5}
&&\Big\|\frac{n_1\big(n_2^2\partial_\mu n_1\partial^\mu n_1 +
  n_1^2\partial_\mu n_2\partial^\mu n_2 - 2n_1n_2\partial_\mu
  n_1\partial^\mu n_2\big)}{1 - n_1^2 - n_2^2}
  \Big\|_{\Gamma, s - 1, L^1}\\\nonumber
&&\leq C(M\epsilon)^5(1 + \tau)^{- 2 + 3\delta}.
\end{eqnarray}
To finish the estimate for $\|f_1(\tau, \cdot)\|_{\Gamma, s - 1,
L^1}$, it remains to bound
\begin{eqnarray}\nonumber
\Big\|\frac{(1 - n_1^2)\partial_\nu n_2 + n_1n_2\partial_\nu n_1}
{\sqrt{1 - n_1^2 - n_2^2}}\partial_\mu\Big(\frac{\partial^\mu
n_1\partial^\nu n_2 - \partial^\nu n_1\partial^\mu n_2}{\sqrt{1 -
n_1^2 - n_2^2}}\Big)\Big\|_{\Gamma, s - 1, L^1}.
\end{eqnarray}
Using H${\rm \ddot{o}}$lder inequality, we can estimate the above
quantity as follows:
\begin{eqnarray}\nonumber
&&\sum_{|\alpha + \beta| \leq s - 1, |\alpha| \leq
  |\beta|}\Big\|\Gamma^\alpha\Big(\frac{1 - n_1^2}{\sqrt{1
  - n_1^2 - n_2^2}}\Big)\Gamma^\beta\Big\{\partial_\nu n_2
  \partial_\mu\Big(\frac{\partial^\mu n_1\partial^\nu n_2
  - \partial^\nu n_1\partial^\mu n_2}{\sqrt{1 - n_1^2
  - n_2^2}}\Big)\Big\}\Big\|_{L^1}\\\nonumber
&&+\ \sum_{|\alpha + \beta| \leq s - 1, |\alpha| \leq
  |\beta|}\Big\|\Gamma^\alpha\Big(\frac{n_1n_2}{\sqrt{1 -
  n_1^2 - n_2^2}}\Big)\Gamma^\beta\Big\{\partial_\nu n_1\partial_\mu
  \Big(\frac{\partial^\mu n_1\partial^\nu n_2 -
  \partial^\nu n_1\partial^\mu n_2}{\sqrt{1 - n_1^2 -
  n_2^2}}\Big)\Big\}\Big\|_{L^1}\\\nonumber
&&+\ C\sum_{|\alpha + \beta| \leq s - 1, |\alpha| > |\beta|}
  \Big\|\Gamma^\alpha\Big(\frac{1 - n_1^2}{\sqrt{1 -
  n_1^2 - n_2^2}}\Big)\Gamma^\beta\Big\{\partial_\nu n_2
  \partial_\mu\Big(\frac{\partial^\mu n_1\partial^\nu n_2
  - \partial^\nu n_1\partial^\mu n_2}{\sqrt{1 - n_1^2
  - n_2^2}}\Big)\Big\}\Big\|_{L^1}\\\nonumber
&&+\ C\sum_{|\alpha + \beta| \leq s - 1, |\alpha| > |\beta|}
  \Big\|\Gamma^\alpha\Big(\frac{n_1n_2}{\sqrt{1 -
  n_1^2 - n_2^2}}\Big)\Gamma^\beta\Big\{\partial_\nu n_1
  \partial_\mu\Big(\frac{\partial^\mu n_1\partial^\nu n_2
  - \partial^\nu n_1\partial^\mu n_2}{\sqrt{1 - n_1^2
  - n_2^2}}\Big)\Big\}\Big\|_{L^1}\\\nonumber
&&\leq C\sum_{i = 1}^2\Big\|\partial_\nu n_i\partial_\mu
  \Big(\frac{\partial^\mu n_1\partial^\nu n_2 - \partial^\nu
  n_1\partial^\mu n_2}{\sqrt{1 - n_1^2 - n_2^2}}\Big)
  \Big\|_{\Gamma, s - 1, L^1}\\\nonumber
&&+\ C\Big\{\Big\|\Gamma\Big(\frac{1 - n_1^2}{\sqrt{1 -
  n_1^2 - n_2^2}}\Big)\Big\|_{\Gamma, s - 2,
  L^2} + \Big\|\Gamma\Big(\frac{n_1n_2}{\sqrt{1 -
  n_1^2 - n_2^2}}\Big)\Big\|_{\Gamma, s - 2,
  L^2}\Big\}\\\nonumber
&&\times\sum_{i = 1}^2\Big\|\partial_\nu n_i\partial_\mu
  \Big(\frac{\partial^\mu n_1\partial^\nu n_2 - \partial^\nu
  n_1\partial^\mu n_2}{\sqrt{1 - n_1^2 - n_2^2}}\Big)
  \Big\|_{\Gamma, [\frac{s - 1}{2}], L^2}.
\end{eqnarray}
By \eqref{E} and Lemma \ref{lemc12}, we hence conclude
\begin{eqnarray}\nonumber
&&\Big\{\Big\|\Gamma\Big(\frac{1 - n_1^2}{\sqrt{1 -
  n_1^2 - n_2^2}}\Big)\Big\|_{\Gamma, s - 2,
  L^2} + \Big\|\Gamma\Big(\frac{n_1n_2}{\sqrt{1 -
  n_1^2 - n_2^2}}\Big)\Big\|_{\Gamma, s - 2,
  L^2}\Big\}\\\nonumber
&&\quad \times\sum_{i = 1}^2\Big\|\partial_\nu n_i\partial_\mu
  \Big(\frac{\partial^\mu n_1\partial^\nu n_2 - \partial^\nu
  n_1\partial^\mu n_2}{\sqrt{1 - n_1^2 - n_2^2}}\Big)
  \Big\|_{\Gamma, [\frac{s - 1}{2}], L^2}\\\nonumber
&&\leq C\|(n_1, n_2)\|_{\Gamma, [\frac{s - 1}{2}],
  L^\infty}\|(n_1, n_2)\|_{\Gamma, s - 1, L^2}\\\nonumber
&&\quad\times \|(\partial n_1, \partial n_2)\|_{\Gamma, s - 2,
  L^\infty}^2\|(\partial n_1, \partial n_2)\|_{\Gamma, s, L^2}\\\nonumber
&&\leq C(M\epsilon)^5(1 + \tau)^{- 1 + 3\delta}.
\end{eqnarray}
On the other hand, using Lemma \ref{lemd11} and Lemma
\ref{lemd12}, one has
\begin{eqnarray}\nonumber
&&\sum_{i = 1}^2\Big\|\partial_\nu n_i\partial_\mu
  \Big(\frac{\partial^\mu n_1\partial^\nu n_2 - \partial^\nu
  n_1\partial^\mu n_2}{\sqrt{1 - n_1^2 - n_2^2}}\Big)
  \Big\|_{\Gamma, s - 1, L^1}\\\nonumber
&&= \sum_{i = 1, 2, |\alpha + \beta| \leq s - 1, |\alpha| \geq
  |\beta|}\Big\|\Gamma^\alpha\partial_\nu n_i\Gamma^\beta\partial_\mu
  \Big(\frac{\partial^\mu n_1\partial^\nu n_2 - \partial^\nu
  n_1\partial^\mu n_2}{\sqrt{1 - n_1^2 - n_2^2}}\Big)
  \Big\|_{L^1}\\\nonumber
&&\quad +\ \sum_{i = 1, 2, |\alpha + \beta| \leq s - 1, |\alpha| <
  |\beta|}\Big\|\Gamma^\alpha\partial_\nu n_i\Gamma^\beta\partial_\mu
  \Big(\frac{\partial^\mu n_1\partial^\nu n_2 - \partial^\nu
  n_1\partial^\mu n_2}{\sqrt{1 - n_1^2 - n_2^2}}\Big)
  \Big\|_{L^1}\\\nonumber
&&\leq C(1 + \tau)^{-1}\Big\{\|(\partial n_1, \partial n_2)
  \|_{\Gamma, s - 1, L^2}\Big(\|(n_1, n_2)\|_{\Gamma, [\frac{s - 1}{2}] + 2, L^\infty}
  \|(\partial n_1, \partial n_2)\|_{\Gamma, [\frac{s - 1}{2}] + 1, L^2}\Big)\\\nonumber
&&\quad +\ \|(\partial n_1, \partial n_2)\|_{\Gamma, [\frac{s -
  1}{2}], L^\infty}\Big(\|(n_1, n_2)\|_{\Gamma, s, L^2}
  \|(\partial n_1, \partial n_2)\|_{\Gamma,
  s, L^2}\Big)\Big\}\\\nonumber
&&\leq C(M\epsilon)^3(1 + \tau)^{- 1 + 3\delta}.
\end{eqnarray}
Thus we obtain
\begin{eqnarray}\label{p3-6}
&&\Big\|\frac{n_1\big(n_2^2\partial_\mu n_1\partial^\mu n_1 +
  n_1^2\partial_\mu n_2\partial^\mu n_2 - 2n_1n_2\partial_\mu
  n_1\partial^\mu n_2\big)}{1 - n_1^2 - n_2^2}
  \Big\|_{\Gamma, s - 1, L^1}\\\nonumber
&&\leq C(M\epsilon)^3[1 + (M\epsilon)^2](1 + \tau)^{- 1 +
  3\delta}.
\end{eqnarray}

Combining \eqref{p3-4}, \eqref{p3-5} and \eqref{p3-6}, we arrive
at
\begin{eqnarray}\label{p3-7}
\|f_1(\tau, \cdot)\|_{\Gamma, s - 1, L^1} \leq C(M\epsilon)^3[1 +
(M\epsilon)^2](1 + \tau)^{- 1 + 3\delta}.
\end{eqnarray}
Inserting \eqref{p3-7} into \eqref{p2-1}, one gets
\begin{eqnarray}\nonumber
&&\|n_1(t, \cdot)\big\|_{\Gamma, s -
  2, L^\infty} \leq C_\star[1 + (M\epsilon)^2]
  (1 + t)^{-\frac{1}{2}}\\\nonumber
&&\quad \times \Big\{\|(n_{10}, n_{20})\|_{H^{s + 2}} +
  \|(n_{11}, n_{21})\|_{H^{s + 1}} + (M\epsilon)^3\Big\}
\end{eqnarray}
for some absolute positive constant $C_\star$. Repeating the above
analysis, one can thus prove
\begin{eqnarray}\nonumber
&&\|n_2(t, \cdot)\|_{\Gamma, s - 2, L^\infty}\leq C_\star[1 +
  (M\epsilon)^2](1 + t)^{-\frac{1}{2}}\\\nonumber
&&\quad \times \Big\{\|(n_{10}, n_{20})\|_{H^{s +
  2}} + \|(n_{11}, n_{21})\|_{H^{s + 1}} + (M\epsilon)^3\Big\}.
\end{eqnarray}
One concludes that the third line in \eqref{Ea} is true provided
that
\begin{eqnarray}\label{p3-8}
\epsilon \leq \frac{1}{2M\sqrt{C_\star}},\quad \|(n_{10},
n_{20})\|_{H^{s + 2}} + \|(n_{11}, n_{21})\|_{H^{s + 1}} \leq
\frac{M\epsilon}{4C_\star}.
\end{eqnarray}

\bigskip

\noindent \textbf{Estimates for $\big\|\big(n_1(t, \cdot), n_1(t,
\cdot)\big)\big\|_{\Gamma, s, L^2}$ in \eqref{Ea}}

\bigskip

By Theorem \ref{thmb22} and using the similar proof as that for
\eqref{p2-1}, one has
\begin{eqnarray}\label{p5-1}
&&\|n_1(t, \cdot)\|_{\Gamma, s, L^2} \leq C(1 + t)^{\frac{1}{2}}
  \big(\|(n_{10}, n_{20})\|_{H^{s +
  2}} + \|(n_{11}, n_{21})\|_{H^{s + 1}}\big)\\\nonumber
&&\quad +\ C(1 + t)^{\frac{1}{2}}\int_0^t\Big(\|f_1(\tau,
  \cdot)\|_{\Gamma, s, L^{\frac{4}{3}}, \chi_1} + (1 + \tau)^{-\frac{1}{2}}
  \|f_1(\tau, \cdot)\|_{\Gamma, s, L^{1, 2}, \chi_2}\Big)d\tau
\end{eqnarray}
and
\begin{eqnarray}\label{p5-2}
&&\|n_2(t, \cdot)\|_{\Gamma, s, L^2} \leq C(1 + t)^{\frac{1}{2}}
  \big(\|(n_{10}, n_{20})\|_{H^{s +
  2}} + \|(n_{11}, n_{21})\|_{H^{s + 1}}\big)\\\nonumber
&&\quad +\ C(1 + t)^{\frac{1}{2}}\int_0^t\Big(\|f_2(\tau,
  \cdot)\|_{\Gamma, s, L^{\frac{4}{3}}, \chi_1} + (1 + \tau)^{-\frac{1}{2}}
  \|f_2(\tau, \cdot)\|_{\Gamma, s, L^{1, 2}, \chi_2}\Big)d\tau,
\end{eqnarray}
where $f_1$ and $f_2$ are given in \eqref{p1-3}. Hence we need to
estimate $\|f_j(\tau, \cdot)\|_{\Gamma, s, L^{\frac{4}{3}},
\chi_1}$ and $\|f_j(\tau, \cdot)\|_{\Gamma, s, L^{1, 2}, \chi_2}$
for $j = 1$, 2.

They can be done as follows:
\begin{eqnarray}\nonumber
&&\|f_1(\tau, \cdot)\|_{\Gamma, s, L^{\frac{4}{3}},
  \chi_1}\\\nonumber
&&\leq  C\Big\|\frac{n_1}{1 - n_1^2 - n_2^2}
  \Big\|_{\Gamma, [\frac{s}{2}], L^\infty}\Big(\big\|
  \partial_\mu n_1\partial^\mu n_1 + \partial_\mu n_2\partial^\mu
  n_2\big\|_{\Gamma, s, L^{\frac{4}{3}}, \chi_1}\\\nonumber
&&+\ \big\|n_2^2\partial_\mu n_1\partial^\mu n_1 +
  n_1^2\partial_\mu n_2\partial^\mu n_2 - 2n_1n_2\partial_\mu
  n_1\partial^\mu n_2\big\|_{\Gamma, s,
  L^{\frac{4}{3}}, \chi_1}\Big)\\\nonumber
&&+\ C\Big\|\frac{n_1}{1 - n_1^2 - n_2^2}
  \Big\|_{\Gamma, s, L^2}\Big(\big\|
  \partial_\mu n_1\partial^\mu n_1 + \partial_\mu n_2\partial^\mu
  n_2\big\|_{\Gamma, [\frac{s}{2}], L^4, \chi_1}\\\nonumber
&&+\ \big\|n_2^2\partial_\mu n_1\partial^\mu n_1 +
  n_1^2\partial_\mu n_2\partial^\mu n_2 - 2n_1n_2\partial_\mu
  n_1\partial^\mu n_2\big\|_{\Gamma, [\frac{s}{2}], L^4, \chi_1}\Big)\\\nonumber
&&+\ C\sum_{i = 1}^2\Big\|\partial_\nu n_i\partial_\mu
  \Big(\frac{\partial^\mu n_1\partial^\nu n_2 - \partial^\nu
  n_1\partial^\mu n_2}{\sqrt{1 - n_1^2 - n_2^2}}\Big)
  \Big\|_{\Gamma, s, L^{\frac{4}{3}}, \chi_1}\\\nonumber
&&+\ C\sum_{|\alpha + \beta| \leq s, |\alpha| > |\beta|}
  \Big\|\Gamma^\alpha\Big(\frac{1 - n_1^2}{\sqrt{1 -
  n_1^2 - n_2^2}}\Big)\Gamma^\beta\Big\{\partial_\nu n_2
  \partial_\mu\Big(\frac{\partial^\mu n_1\partial^\nu n_2
  - \partial^\nu n_1\partial^\mu n_2}{\sqrt{1 - n_1^2
  - n_2^2}}\Big)\Big\}\Big\|_{L^{\frac{4}{3}}, \chi_1}\\\nonumber
&&+\ C\sum_{|\alpha + \beta| \leq s, |\alpha| > |\beta|}
  \Big\|\Gamma^\alpha\Big(\frac{n_1n_2}{\sqrt{1 -
  n_1^2 - n_2^2}}\Big)\Gamma^\beta\Big\{\partial_\nu n_1
  \partial_\mu\Big(\frac{\partial^\mu n_1\partial^\nu n_2
  - \partial^\nu n_1\partial^\mu n_2}{\sqrt{1 - n_1^2
  - n_2^2}}\Big)\Big\}\Big\|_{L^{\frac{4}{3}}, \chi_1}.
\end{eqnarray}
Next , we use Lemma \ref{lemc12} and Lemma \ref{lema32}, to obtain
\begin{eqnarray}\nonumber
&&\Big\|\frac{n_1}{1 - n_1^2 - n_2^2}
  \Big\|_{\Gamma, [\frac{s}{2}], L^\infty}\Big(\big\|
  \partial_\mu n_1\partial^\mu n_1 + \partial_\mu n_2\partial^\mu
  n_2\big\|_{\Gamma, s, L^{\frac{4}{3}}, \chi_1}\\\nonumber
&&+\ \big\|n_2^2\partial_\mu n_1\partial^\mu n_1 +
  n_1^2\partial_\mu n_2\partial^\mu n_2 - 2n_1n_2\partial_\mu
  n_1\partial^\mu n_2\big\|_{\Gamma, s,
  L^{\frac{4}{3}}, \chi_1}\Big)\\\nonumber
&&+\ \Big\|\frac{n_1}{1 - n_1^2 - n_2^2}
  \Big\|_{\Gamma, s, L^2}\Big(\big\|
  \partial_\mu n_1\partial^\mu n_1 + \partial_\mu n_2\partial^\mu
  n_2\big\|_{\Gamma, [\frac{s}{2}], L^4, \chi_1}\\\nonumber
&&+\ \big\|n_2^2\partial_\mu n_1\partial^\mu n_1 +
  n_1^2\partial_\mu n_2\partial^\mu n_2 - 2n_1n_2\partial_\mu
  n_1\partial^\mu n_2\big\|_{\Gamma, [\frac{s}{2}], L^4, \chi_1}\Big)\\\nonumber
&&\leq C\|(n_1, n_2)\|_{\Gamma, [\frac{s}{2}], L^\infty}\big[
  \|(\partial n_1, \partial n_2)\|_{\Gamma, s,
  L^2}\|(\partial n_1, \partial n_2)\|_{\Gamma,
  [\frac{s}{2}], L^4, \chi_1}\\\nonumber
&&\quad +\ \|(n_1, n_2)\|_{\Gamma, s,
  L^2}\|(\partial n_1, \partial n_2)\|_{\Gamma,
  [\frac{s}{2}], L^4, \chi_1}\|(n_1, n_2)\|_{\Gamma,
  [\frac{s}{2}] + 1, L^\infty}^2\big]\\\nonumber
&&\quad +\ C(1 + \tau)^{-\frac{1}{2}}\|(n_1, n_2)\|_{\Gamma, s,
  L^2}\Big(\big\|\partial_\mu n_1\partial^\mu n_1 + \partial_\mu n_2
   \partial^\mu n_2\big\|_{\Gamma, [\frac{s}{2}] + 1, L^2}\\\nonumber
&&\quad +\ \|(n_1, n_2)\|_{\Gamma, [\frac{s}{2}] + 1, L^\infty}^2
   \sum_{i, j = 1}^2\|\partial^\mu n_i\partial_\mu n_j
   \|_{\Gamma, [\frac{s}{2}] + 1, L^2}\Big)\\\nonumber
&&\leq C(1 + \tau)^{-\frac{1}{2}}\Big\{\|(n_1, n_2)\|_{\Gamma,
  [\frac{s}{2}], L^\infty}\|(\partial n_1, \partial n_2)\|_{\Gamma, s,
  L^2}\|(\partial n_1, \partial n_2)\|_{\Gamma, [\frac{s}{2}] + 1, L^2}\\\nonumber
&&\quad +\ (1 + \tau)^{-1}\|(n_1, n_2)\|_{\Gamma, s, L^2}
  \|(n_1, n_2)\|_{\Gamma, [\frac{s}{2}] + 2, L^\infty}
  \|(\partial n_1, \partial n_2)\|_{\Gamma, [\frac{s}{2}] + 1, L^2}\Big\}\\\nonumber
&&\leq C(M\epsilon)^3(1 + \tau)^{-1 + 2\delta}.
\end{eqnarray}
In a similar manner, one deduces that
\begin{eqnarray}\nonumber
&&\sum_{i = 1}^2\Big\|\partial_\nu n_i\partial_\mu
  \Big(\frac{\partial^\mu n_1\partial^\nu n_2 - \partial^\nu
  n_1\partial^\mu n_2}{\sqrt{1 - n_1^2 - n_2^2}}\Big)
  \Big\|_{\Gamma, s, L^{\frac{4}{3}}, \chi_1}\\\nonumber
&&\leq C\|(\partial n_1, \partial n_2)\|_{\Gamma, s,
  L^2}\Big\|\frac{\partial^\mu n_1\partial^\nu n_2 - \partial^\nu
  n_1\partial^\mu n_2}{\sqrt{1 - n_1^2 - n_2^2}}
  \Big\|_{\Gamma, [\frac{s}{2}] + 1, L^4, \chi_1}\\\nonumber
&&\quad +\ C\|(\partial n_1, \partial n_2)\|_{\Gamma,
  [\frac{s}{2}], L^4, \chi_1}\Big\|\partial\Big(\frac{\partial^\mu n_1
  \partial^\nu n_2 - \partial^\nu n_1\partial^\mu n_2}{\sqrt{1
  - n_1^2 - n_2^2}}\Big)\Big\|_{\Gamma, s, L^2}\\\nonumber
&&\leq C(1 + \tau)^{-\frac{1}{2}}\big[\|(\partial n_1,
  \partial n_2)\|_{\Gamma, s, L^2} + \|(\partial^2 n_1, \partial^2
  n_2)\|_{\Gamma, s, L^2}\big]\|\partial n_1\partial n_2\|_{\Gamma,
  [\frac{s}{2}] + 2, L^2}\\\nonumber
&&\leq C(M\epsilon)^3(1 + \tau)^{-1 + 2\delta},
\end{eqnarray}
and that
\begin{eqnarray}\nonumber
 &&\sum_{|\alpha + \beta| \leq s, |\alpha| > |\beta|}
  \Big\|\Gamma^\alpha\Big(\frac{1 - n_1^2}{\sqrt{1 -
  n_1^2 - n_2^2}}\Big)\Gamma^\beta\Big\{\partial_\nu n_2
  \partial_\mu\Big(\frac{\partial^\mu n_1\partial^\nu n_2
  - \partial^\nu n_1\partial^\mu n_2}{\sqrt{1 - n_1^2
  - n_2^2}}\Big)\Big\}\Big\|_{L^{\frac{4}{3}}, \chi_1}\\\nonumber
&&+\ C\sum_{|\alpha + \beta| \leq s, |\alpha| > |\beta|}
  \Big\|\Gamma^\alpha\Big(\frac{n_1n_2}{\sqrt{1 -
  n_1^2 - n_2^2}}\Big)\Gamma^\beta\Big\{\partial_\nu n_1
  \partial_\mu\Big(\frac{\partial^\mu n_1\partial^\nu n_2
  - \partial^\nu n_1\partial^\mu n_2}{\sqrt{1 - n_1^2
  - n_2^2}}\Big)\Big\}\Big\|_{L^{\frac{4}{3}}, \chi_1}\\\nonumber
&&\leq C\|(n_1, n_2)\|_{\Gamma, [\frac{s}{2}], L^\infty}\|(n_1,
  n_2)\|_{\Gamma, s, L^2}\Big\|\partial_\nu n_2
  \partial_\mu\Big(\frac{\partial^\mu n_1\partial^\nu n_2
  - \partial^\nu n_1\partial^\mu n_2}{\sqrt{1 - n_1^2
  - n_2^2}}\Big)\Big\|_{\Gamma, [\frac{s}{2}], L^4, \chi_1}\\\nonumber
&&\leq C(1 + \tau)^{-\frac{1}{2}}\|(n_1, n_2)\|_{\Gamma,
  [\frac{s}{2}], L^\infty}\|(n_1, n_2)\|_{\Gamma, s, L^2}\\\nonumber
&&\quad \times \Big\|\partial_\nu n_2
  \partial_\mu\Big(\frac{\partial^\mu n_1\partial^\nu n_2
  - \partial^\nu n_1\partial^\mu n_2}{\sqrt{1 - n_1^2
  - n_2^2}}\Big)\Big\|_{\Gamma, [\frac{s}{2}] + 1,
  L^2}\\\nonumber
&&\leq C(M\epsilon)^5(1 + \tau)^{- \frac{3}{2} + 2\delta}.
\end{eqnarray}
Therefore, we have
\begin{eqnarray}\label{p5-3}
\|f_1(\tau, \cdot)\|_{\Gamma, s, L^{\frac{4}{3}}, \chi_1} \leq
C(M\epsilon)^3(1 + \tau)^{-1 + 2\delta}.
\end{eqnarray}

Next, by  H${\rm \ddot{o}}$lder inequality and Theorem
\ref{thma21}, one can proceed as follows:
\begin{eqnarray}\nonumber
&&\|f_1(\tau, \cdot)\|_{\Gamma, s, L^{1, 2}, \chi_2}\\\nonumber
&&\leq C\Big\|\frac{n_1}{1 - n_1^2 - n_2^2}
  \Big\|_{\Gamma, [\frac{s}{2}], L^\infty}\Big(\big\|
  \partial_\mu n_1\partial^\mu n_1 + \partial_\mu n_2\partial^\mu
  n_2\big\|_{\Gamma, s, L^{1, 2}}\\\nonumber
&&+\ \big\|n_2^2\partial_\mu n_1\partial^\mu n_1 +
  n_1^2\partial_\mu n_2\partial^\mu n_2 - 2n_1n_2\partial_\mu
  n_1\partial^\mu n_2\big\|_{\Gamma, s,
  L^{1, 2}}\Big)\\\nonumber
&&+\ C\Big\|\frac{n_1}{1 - n_1^2 - n_2^2}
  \Big\|_{\Gamma, s, L^2}\Big(\big\|
  \partial_\mu n_1\partial^\mu n_1 + \partial_\mu n_2\partial^\mu
  n_2\big\|_{\Gamma, [\frac{s}{2}], L^{2, \infty}}\\\nonumber
&&+\ \big\|n_2^2\partial_\mu n_1\partial^\mu n_1 +
  n_1^2\partial_\mu n_2\partial^\mu n_2 - 2n_1n_2\partial_\mu
  n_1\partial^\mu n_2\big\|_{\Gamma, [\frac{s}{2}], L^{2, \infty}}\Big)\\\nonumber
&&+\ C\sum_{i = 1}^2\Big\|\partial_\nu n_i\partial_\mu
  \Big(\frac{\partial^\mu n_1\partial^\nu n_2 - \partial^\nu
  n_1\partial^\mu n_2}{\sqrt{1 - n_1^2 - n_2^2}}\Big)
  \Big\|_{\Gamma, s, L^{1, 2}}\\\nonumber
&&+\ C\Big\{\Big\|\Gamma\Big(\frac{1 - n_1^2}{\sqrt{1 -
  n_1^2 - n_2^2}}\Big)\Big\|_{\Gamma, s - 1, L^2} + \Big\|\Gamma\Big(\frac{n_1n_2}{\sqrt{1 -
  n_1^2 - n_2^2}}\Big)|\Big\|_{\Gamma, s - 1, L^2}\Big\}\\\nonumber
&&\quad \times \sum_{i = 1}^2\Big\|\partial_\nu
  n_i\partial_\mu\Big(\frac{\partial^\mu n_1\partial^\nu n_2
  - \partial^\nu n_1\partial^\mu n_2}{\sqrt{1 - n_1^2
  - n_2^2}}\Big)\Big\|_{\Gamma, [\frac{s}{2}], L^{1, \infty}}.
\end{eqnarray}
Noting that
\begin{equation}\nonumber
\begin{cases}
\big\|\partial_\mu n_1\partial^\mu n_1 + \partial_\mu
  n_2\partial^\mu n_2\big\|_{\Gamma, s, L^{1, 2}}\\
\quad\quad  \leq  C\|(\partial n_1, \partial n_2)\|_{\Gamma, s,
  L^2}\|(\partial n_1, \partial n_2)\|_{\Gamma, [\frac{s}{2}], L^{2,
  \infty}}\\
\quad\quad  \leq C\|(\partial n_1, \partial n_2)\|_{\Gamma, s,
  L^2}\|(\partial n_1, \partial n_2)
  \|_{\Gamma, [\frac{s}{2}] + 1, L^2},\\[-4mm]\\
\big\|\partial_\mu n_1\partial^\mu n_1 + \partial_\mu
  n_2\partial^\mu n_2\big\|_{\Gamma, [\frac{s}{2}], L^{2,
  \infty}}\\
\quad\quad \leq C\big\|\partial_\mu n_1\partial^\mu n_1 +
  \partial_\mu n_2\partial^\mu n_2\big\|_{\Gamma,
  [\frac{s}{2}] + 1, L^2}
\end{cases}
\end{equation}
and using Theorem \ref{thma21}, one can further estimate
\begin{eqnarray}\label{p5-4}
\|f_1(\tau, \cdot)\|_{\Gamma, s, L^{1, 2}, \chi_2} \leq
C(M\epsilon)^3(1 +  \tau)^{-\frac{1}{2} + 2\delta}.
\end{eqnarray}

By inserting \eqref{p5-3} into \eqref{p5-1}, one hence conclude
\begin{eqnarray}\label{p5-5}
&&\|n_1(t, \cdot)\|_{\Gamma, s, L^2} \leq C_\star\big[(1 +
  t)^{\frac{1}{2}}\big(\|(n_{10}, n_{20})\|_{H^{s +
  2}}\\\nonumber
&&\quad +\ \|(n_{11}, n_{22})\|_{H^{s + 1}}\big) + (M\epsilon)^3(1
+ t)^{\frac{1}{2} + 2\delta}\big].
\end{eqnarray}
Repeating the above analysis, one has
\begin{eqnarray}\label{p5-6}
&&\|n_2(t, \cdot)\|_{\Gamma, s, L^2} \leq C_\star\big[(1 +
  t)^{\frac{1}{2}}\big(\|(n_{10}, n_{20})\|_{H^{s +
  2}}\\\nonumber
&&\quad +\ \|(n_{11}, n_{22})\|_{H^{s + 1}}\big) + (M\epsilon)^3(1
+ t)^{\frac{1}{2} + 2\delta}\big].
\end{eqnarray}
By \eqref{p5-5} and \eqref{p5-6}, we see that the second inequality in
\eqref{Ea} is true provided that \eqref{p3-8} is satisfied.

\section{Energy Estimates}

This section is devoted to estimating $\sum_{i =
1}^2\big\|\partial^i\textbf{n}(t, \cdot)\big\|_{\Gamma, s, L^2}$
and proving the first inequality in \eqref{Ea}. We begin with the
following Lemma:

\begin{lem}\label{lemc13}
Let $n \geq 2$ and
\begin{equation}\nonumber
{\rm supp}\ v, {\rm supp}\ w \subset \{(t, x): |x| \leq t +
\rho\}.
\end{equation}
Then for all $t \geq 0$:
\begin{equation}\nonumber
\|v\partial w(t, \cdot)\|_{L^2} \leq C_\rho\|\nabla v(t,
\cdot)\|_{L^2}\|\Gamma w(t, \cdot)\|_{L^\infty}.
\end{equation}
\end{lem}
\begin{proof}
By \eqref{a14} in Proposition \ref{propa12}, we have
\begin{eqnarray}\nonumber
&&\|v\partial w(t, \cdot)\|_{L^2} \leq C_\rho\Big\|\frac{v\Gamma
  w(t, \cdot)}{\rho + \big|t - |x|\big|}\Big\|_{L^2}\\\nonumber
&&\leq C_\rho\Big\|\frac{v}{\rho + \big|t - |x|\big|}
  \Big\|_{L^2}\|\Gamma w(t, \cdot)\|_{L^\infty}\\\nonumber
&&\leq C_\rho\|\nabla v\|_{L^2}\|\Gamma w(t, \cdot)\|_{L^\infty}.
\end{eqnarray}
Here we used the following Hardy's inequality
\begin{eqnarray}\nonumber
&&\Big\|\frac{v}{\rho + \big|t - |x|\big|}\Big\|_{L^2}^2
  \leq C_\rho\int_{|\xi| = 1}\int_0^{t + \rho}\frac{|v(r\xi)|^2}{\big(2\rho
  + t - r\big)^2} r^{n - 1}drdS\\\nonumber
&&= C_\rho\int_{|\xi| = 1}\int_0^{t + \rho}|v(r\xi)|^2
  r^{n - 1}d\frac{1}{\big(2\rho + t - r\big)}dS\\\nonumber
&&= - C_\rho\int_{|\xi| = 1}\int_0^{t + \rho}
  \frac{|v(r\xi)|^2}{2\rho + t - r} dr^{n - 1}dS\\\nonumber
&&\quad -\ C_\rho\int_{|\xi| = 1}\int_0^{t + \rho}
  \frac{2v(r\xi)v_r(r\xi)}{2\rho + t - r} r^{n - 1}drdS\\\nonumber
&&\leq - C_\rho\int_{|\xi| = 1}\int_0^{t + \rho}
  \frac{2v(r\xi)v_r(r\xi)}{2\rho + t - r} r^{n - 1}drdS\\\nonumber
&&\leq C_\rho\|\nabla v\|_{L^2}\Big\|\frac{v}{\rho + \big|t -
  |x|\big|}\Big\|_{L^2}.
\end{eqnarray}
\end{proof}

Now let us rewrite the Fadeev model \eqref{Faddeev-1} as
\begin{eqnarray}\label{Faddeev-3}
&&\partial_\mu\partial^\mu\textbf{n} +
(\partial_\mu\textbf{n}\cdot\partial^\mu\textbf{n})\textbf{n} +
\big[\partial_\mu\big(\textbf{n}\cdot[\partial^\mu\textbf{n}
\wedge \partial^\nu
\textbf{n}]\big)\big]\partial_\nu\textbf{n}\wedge \textbf{n} = 0.
\end{eqnarray}
For $|\alpha| \leq s$ and $i = 0$, 1, similarly as in
\eqref{p1-4}, one derives from \eqref{Faddeev-3} that
\begin{equation}\nonumber
\Box\partial^i\Gamma^\alpha \textbf{n} = - \sum_{\beta \leq
\alpha} C_{\alpha\beta}\partial^i\Gamma^\beta\Big\{
(\partial_\mu\textbf{n}\cdot\partial^\mu\textbf{n})\textbf{n} +
\big[\partial_\mu\big(\textbf{n}\cdot[\partial^\mu\textbf{n}\wedge
\partial^\nu\textbf{n}]\big)\big]\partial_\nu\textbf{n}\wedge
\textbf{n}\Big\}.
\end{equation}
For $i = 0$ and 1, taking the $L^2$ inner product of the above
equations with $\partial_t\partial^i\Gamma^\alpha \textbf{n}$
respectively and then adding them together, one has
\begin{eqnarray}\label{h1}
&&\frac{1}{2}\frac{d}{dt}\sum_{i = 0}^1\sum_{|\alpha| \leq s}
  \|(\partial_t, \nabla)\partial^i\Gamma^\alpha \textbf{n}\|_{L^2}^2\\\nonumber
&&= - \sum_{i = 0}^1\sum_{|\alpha| \leq s}\sum_{\beta
  \leq \alpha}C_{\alpha\beta}\Big\{\int\partial_t\partial^i\Gamma^\alpha
  \textbf{n}\cdot\partial^i\Gamma^\beta\big[(\partial_\mu\textbf{n}
  \cdot\partial^\mu\textbf{n})\textbf{n}\big]dx\\\nonumber
&&+\ \int\partial_t\partial^i\Gamma^\alpha \textbf{n}\cdot
  \partial^i\Gamma^\beta\Big[\big[\partial_\mu\big(\textbf{n}\cdot[
  \partial^\mu\textbf{n}\wedge\partial^\nu\textbf{n}]\big)\big]
  \partial_\nu\textbf{n}\wedge\textbf{n}\Big]dx\Big\}.
\end{eqnarray}

Let us estimate the first term on the right hand side of
\eqref{h1}. A straightforward calculation gives
\begin{eqnarray}\nonumber
&&- \sum_{i = 0}^1\sum_{|\alpha| \leq s}\sum_{\beta
  \leq \alpha}C_{\alpha\beta}\int\partial_t\partial^i\Gamma^\alpha
  \textbf{n}\cdot\partial^i\Gamma^\beta\big[(\partial_\mu\textbf{n}
  \cdot\partial^\mu\textbf{n})\textbf{n}\big]dx\\\nonumber
&&\leq C\sum_{i = 0}^1\sum_{|\alpha| \leq s}
  \sum_{\beta \leq \alpha}\big\|\partial_t\partial^i\Gamma^\alpha
  \textbf{n}\cdot\textbf{n}\big\|_{L^2}\big\|\partial^i\Gamma^\beta
  (\partial_\mu\textbf{n}\cdot\partial^\mu\textbf{n})\big\|_{L^2}\\\nonumber
&&\quad +\ C\sum_{i = 0}^1\sum_{|\alpha| \leq s}\sum_{\gamma \leq
  \beta \leq \alpha, j \leq i, j + |\gamma| \geq 1}\big\|\partial_t\partial^i
  \Gamma^\alpha\textbf{n}\big\|_{L^2}\\\nonumber
&&\quad \Big\{\sum_{j + |\gamma| \geq i - j + |\beta -
  \gamma|}\big\|\partial^j\Gamma^\gamma\textbf{n}\big\|_{L^2}
  \big\|\partial^{i - j}\Gamma^{\beta - \gamma}(\partial_\mu
  \textbf{n}\cdot\partial^\mu\textbf{n})\big\|_{L^\infty}\\\nonumber
&&\quad +\ \sum_{j + |\gamma| < i - j + |\beta - \gamma|
  }\big\|\partial^j\Gamma^\gamma\textbf{n}\big\|_{L^\infty}
  \big\|\partial^{i - j}\Gamma^{\beta - \gamma}
  (\partial_\mu \textbf{n}\cdot\partial^\mu
  \textbf{n})\big\|_{L^2}\Big\}.
\end{eqnarray}
Noting the null structure of the nonlinearity and using \eqref{E},
Lemma \ref{lemd11} and Lemma \ref{lemd12}, we compute
\begin{eqnarray}\label{h2}
&&\sum_{i = 0}^1\sum_{|\alpha| \leq s}\sum_{\gamma \leq
  \beta \leq \alpha, j \leq i, j + |\gamma| \geq 1}\big\|\partial_t\partial^i
  \Gamma^\alpha\textbf{n}\big\|_{L^2}\\\nonumber
&&\quad \Big\{\sum_{j + |\gamma| \geq i - j + |\beta -
  \gamma|}\big\|\partial^j\Gamma^\gamma\textbf{n}\big\|_{L^2}
  \big\|\partial^{i - j}\Gamma^{\beta - \gamma}(\partial_\mu
  \textbf{n}\cdot\partial^\mu\textbf{n})\big\|_{L^\infty}\\\nonumber
&&\quad +\ \sum_{j + |\gamma| < i - j + |\beta - \gamma|
  }\big\|\partial^j\Gamma^\gamma\textbf{n}\big\|_{L^\infty}
  \big\|\partial^{i - j}\Gamma^{\beta - \gamma}
  (\partial_\mu \textbf{n}\cdot\partial^\mu
  \textbf{n})\big\|_{L^2}\Big\}.\\\nonumber
&&\leq C(M\epsilon)^4(1 + t)^\delta(1 + t)^{\frac{1}{2} +
  2\delta}(1 + t)^{-2} + C(M\epsilon)^4(1 + t)^\delta(1 + t)^{-
  1 + \delta}\\\nonumber
&&\leq C(M\epsilon)^4(1 + t)^{- 1 + 2\delta}.
\end{eqnarray}
On the other hand, by $\textbf{n}\cdot \textbf{n} = 1$ (which
means $\textbf{n}\cdot \textbf{n}_t = 0$), one has
\begin{eqnarray}\label{h3}
&&\textbf{n}\cdot\partial^i\Gamma^\alpha\partial_t\textbf{n} =
  -\big[\partial^i\Gamma^\alpha (\textbf{n}_t\cdot\textbf{n}) -
  \textbf{n}\cdot\partial^i\Gamma^\alpha
  \partial_t\textbf{n}\big]\\\nonumber
&&= - \sum_{0 \leq j \leq i, \beta \leq \alpha, j + |\beta| \geq
  1}C_{j\beta}\partial^{i - j}\Gamma^{\alpha - \beta}
  \textbf{n}_t\cdot\partial^j\Gamma^\beta\textbf{n}.
\end{eqnarray}
Consequently, a similar argument as in \eqref{h2} gives
\begin{eqnarray}\nonumber
&&\sum_{i = 0}^1\sum_{|\alpha| \leq s}
  \sum_{\beta \leq \alpha}\big\|\partial_t\partial^i\Gamma^\alpha
  \textbf{n}\cdot\textbf{n}\big\|_{L^2}\big\|\partial^i\Gamma^\beta
  (\partial_\mu\textbf{n}\cdot\partial^\mu\textbf{n})\big\|_{L^2}\\\nonumber
&&\leq C(M\epsilon)^2(1 + t)^{- \frac{1}{2} + \delta}\sum_{i
  = 0}^1\sum_{|\alpha| \leq s}\sum_{\beta \leq \alpha, |\beta| \geq s
  - 2}\big\|\partial^{i}\Gamma^{\alpha - \beta}\textbf{n}_t\cdot
  \Gamma^\beta\textbf{n}\big\|_{L^2}\\\nonumber
&&\quad +\ C(M\epsilon)^4(1 + t)^{- 1 + 2\delta}.
\end{eqnarray}
Using Lemma \ref{lemc13}, one can bound the right hand side of
above equality by
\begin{eqnarray}\nonumber
C(M\epsilon)^4(1 + t)^{- 1 + 2\delta}.
\end{eqnarray}
Finally, one has
\begin{eqnarray}\label{h4}
- \sum_{i = 0}^1\sum_{|\alpha| \leq s}\sum_{\beta \leq
\alpha}C_{\alpha\beta}\int\partial_t\partial^i\Gamma^\alpha
\textbf{n}\cdot\partial^i\Gamma^\beta\big[(\partial_\mu\textbf{n}
\cdot\partial^\mu\textbf{n})\textbf{n}\big]dx \leq
C(M\epsilon)^4(1 + t)^{- 1 + 2\delta}.
\end{eqnarray}

To estimate the right hand side of \eqref{h1}, it remains to bound
\begin{eqnarray}\nonumber
- \sum_{i = 0}^1\sum_{|\alpha| \leq s}\sum_{\beta \leq
\alpha}C_{\alpha\beta}\Big\{\int\partial_t\partial^i\Gamma^\alpha
\textbf{n}\cdot\partial^i\Gamma^\beta\Big[\big[\partial_\mu
\big(\textbf{n}\cdot[
\partial^\mu\textbf{n}\wedge \partial^\nu\textbf{n}]\big)\big]
  \partial_\nu\textbf{n}\wedge\textbf{n}\Big]dx\Big\}.
\end{eqnarray}
By a similar argument as \eqref{h4}, one can bound the above
quantity by
\begin{eqnarray}\nonumber
- \sum_{i = 0}^1\int\partial_t\partial^i\Gamma^s\textbf{n}
  \cdot(\partial_\nu\textbf{n}\wedge\textbf{n})
  \partial_\mu\big(\textbf{n}\cdot\partial^i\Gamma^s[\partial^\mu
  \textbf{n}\wedge\partial^\nu \textbf{n}]\big)dx + C(M\epsilon)^4(1 + t)^{- 1 + 2\delta},
\end{eqnarray}
which is equal to
\begin{eqnarray}\nonumber
&&\sum_{i = 0}^1\int\partial_t\partial^i\Gamma^s\textbf{n}
  \cdot(\nabla\textbf{n}\wedge\textbf{n})
  \partial_t\big(\textbf{n}\cdot\partial^i\Gamma^s[\partial_t
  \textbf{n}\wedge\nabla \textbf{n}]\big)dx\\\nonumber
&&+ \sum_{i = 0}^1\int\partial_t\partial^i\Gamma^s\textbf{n}
  \cdot(\partial_t\textbf{n}\wedge\textbf{n})
  \nabla\big(\textbf{n}\cdot\partial^i\Gamma^s[\nabla
  \textbf{n}\wedge\partial_t \textbf{n}]\big)dx\\\nonumber
&&- \sum_{i = 0}^1\int\partial_t\partial^i\Gamma^s\textbf{n}
  \cdot(\partial_1\textbf{n}\wedge\textbf{n})
  \partial_2\big(\textbf{n}\cdot\partial^i\Gamma^s[\partial_2
  \textbf{n}\wedge\partial_1 \textbf{n}]\big)dx\\\nonumber
&&- \sum_{i = 0}^1\int\partial_t\partial^i\Gamma^s\textbf{n}
  \cdot(\partial_2\textbf{n}\wedge\textbf{n})
  \partial_1\big(\textbf{n}\cdot\partial^i\Gamma^s[\partial_1
  \textbf{n}\wedge\partial_2 \textbf{n}]\big)dx\\\nonumber
&&+\ C(M\epsilon)^4(1 + t)^{- 1 + 2\delta}.
\end{eqnarray}
Now let us rewrite the above quantity as
\begin{eqnarray}\label{h5}
&&\sum_{i = 0}^1\int\big((\partial_t\partial^i\Gamma^s
  \textbf{n}\wedge\nabla\textbf{n})\cdot\textbf{n}\big)
  \partial_t\big(\textbf{n}\cdot\partial^i\Gamma^s[\partial_t
  \textbf{n}\wedge\nabla \textbf{n}]\big)dx\\\nonumber
&&- \sum_{i = 0}^1\int\nabla\big((\partial_t\textbf{n}
  \wedge\partial_t\partial^i\Gamma^s\textbf{n})\cdot\textbf{n}\big)
  \big(\textbf{n}\cdot\partial^i\Gamma^s[\partial_t
  \textbf{n}\wedge\nabla \textbf{n}]\big)dx\\\nonumber
&&+ \sum_{i = 0}^1\int\partial_2\big((\partial_1\textbf{n}\wedge
  \partial_t\partial^i\Gamma^s\textbf{n})\cdot\textbf{n}\big)
  \big(\textbf{n}\cdot\partial^i\Gamma^s[\partial_1
  \textbf{n}\wedge\partial_2 \textbf{n}]\big)dx\\\nonumber
&&+ \sum_{i = 0}^1\int\partial_1\big((\partial_t\partial^i
  \Gamma^s\textbf{n}\wedge\partial_2\textbf{n})\cdot\textbf{n}\big)
  \big(\textbf{n}\cdot\partial^i\Gamma^s[\partial_1
  \textbf{n}\wedge\partial_2 \textbf{n}]\big)dx\\\nonumber
&&+\ C(M\epsilon)^4(1 + t)^{- 1 + 2\delta}.
\end{eqnarray}
Similarly as in \eqref{h4}, one can estimate \eqref{h5} by
\begin{eqnarray}\label{h6}
&&\sum_{i = 0}^1\frac{1}{2}\frac{d}{dt}\Big(\big\|
  \textbf{n}\cdot\partial^i\Gamma^s[\partial_t
  \textbf{n}\wedge\nabla \textbf{n}]\big\|_{L^2}^2 + \big\|
  \textbf{n}\cdot\partial^i\Gamma^s[\partial_1
  \textbf{n}\wedge\partial_2 \textbf{n}]\big\|_{L^2}^2\\\nonumber
&&-\ 2\int\big((\partial_t\textbf{n}
  \wedge\partial^i\Gamma^s\nabla\textbf{n})\cdot\textbf{n}\big)
  \big(\textbf{n}\cdot\partial^i\Gamma^s[\partial_t
  \textbf{n}\wedge\nabla \textbf{n}]\big)dx\Big)
  + C(M\epsilon)^4(1 + t)^{- 1 + 2\delta}\\\nonumber
&&\leq \sum_{i = 0}^1\frac{1}{2}\frac{d}{dt}\Big(\big\|
  \textbf{n}\cdot\partial^i\Gamma^s[\partial_1
  \textbf{n}\wedge\partial_2 \textbf{n}]\big\|_{L^2}^2
  + \big\|(\partial^i\Gamma^s\partial_t\textbf{n}\wedge\nabla
  \textbf{n})\cdot\textbf{n}\big\|_{L^2}^2\\\nonumber
&&-\ \big\|(\partial_t\textbf{n} \wedge\partial^i\Gamma^s
  \nabla\textbf{n})\cdot\textbf{n}\big\|_{L^2}^2\Big)
  + C(M\epsilon)^4(1 + t)^{- 1 + 2\delta}.
\end{eqnarray}
Inserting \eqref{h4} and \eqref{h6} into \eqref{h1}, we finally
arrive at
\begin{eqnarray}\label{h7}
&&\frac{1}{2}\frac{d}{dt}\sum_{i = 0}^1\Big(\sum_{|\alpha| \leq s}
  \|(\partial_t, \nabla)\partial^i\Gamma^\alpha \textbf{n}\|_{L^2}^2
  - \big\|\textbf{n}\cdot\partial^i\Gamma^s[\partial_1
  \textbf{n}\wedge\partial_2 \textbf{n}]\big\|_{L^2}^2\\\nonumber
&&-\ \big\|(\partial^i\Gamma^s\partial_t\textbf{n}\wedge\nabla
  \textbf{n})\cdot\textbf{n}\big\|_{L^2}^2
 + \big\|(\partial_t\textbf{n} \wedge\partial^i\Gamma^s
  \nabla\textbf{n})\cdot\textbf{n}\big\|_{L^2}^2\Big)\\\nonumber
&&\leq  C(M\epsilon)^4(1 + t)^{- 1 + 2\delta}.
\end{eqnarray}
This completes the proof of the first inequality in \eqref{Ea}.

\section*{Acknowledgement}
Zhen Lei was in part supported by NSFC (grants No. 10801029 and
10911120384), FANEDD, Shanghai Rising Star Program (10QA1400300),
SGST 09DZ2272900 and SRF for ROCS, SEM. Yi Zhou was partially
supported by the NSFC grant 10728101, the 973 project of the
Ministry of Science and Technology of China, the Doctoral Program
Foundation of the Ministry of Education of China, the "111"
project (B08018) and SGST 09DZ2272900. Fanghua Lin is partially
supported by an NSF grant. Part of the work was carried out when
Zhen Lei was visiting the Courant Institute.


\end{document}